\documentclass[12pt]{article}

\usepackage{amsmath,amsfonts,amssymb}
\usepackage{amsthm}
\usepackage{graphicx}

\usepackage[colorlinks=true,citecolor=black,linkcolor=black,urlcolor=blue]{hyperref}

\date{Jan 21, 2019}

\title{The Ramsey Number of Fano Plane \\ Versus Tight Path}

\author{ \qquad J\'ozsef Balogh \thanks{and Moscow Institute of Physics and Technology, 9 Institutskiy per., Dolgoprodny, Moscow Region,141701, Russian Federation. Research is partially supported by NSF Grant DMS-1500121, Arnold O. Beckman Research
Award (UIUC Campus Research Board RB 18132) and the Langan Scholar Fund (UIUC).} \qquad Felix Christian Clemen \qquad  \\
\small Department of Mathematics\\[-0.8ex]
\small University of Illinois at Urbana-Champaign\\[-0.8ex] 
\small  Urbana, Illinois 61801, U.S.A.\\
\small\tt jobal@math.uiuc.edu \qquad fclemen2@illinois.edu 
\and Jozef Skokan\\
\small Department of Mathematics\\[-0.8ex]
\small London School of Economics\\[-0.8ex]
\small London, U.K.\\
\small\tt j.skokan@lse.ac.uk\\
\and Adam Zsolt Wagner \thanks{Research has been partially performed while at the University of Illinois at Urbana-Champaign.}\\
\small Department of Mathematics\\[-0.8ex]
\small ETH\\[-0.8ex]
\small Z\"urich, Switzerland \\
\small\tt zsolt.wagner@math.ethz.ch 
  }

\def\Ci{C_n^{(3)}}
\def\FF{\mathbb{F}}
\def\Ti{P_n^t}
\def\HH{\mathcal{H}}
\def\eps{\varepsilon}

\newtheorem{theorem}{Theorem}
\newtheorem{lemma}{Lemma}

\newtheorem{definition}{Definition}

\begin{document}

 \maketitle

\begin{abstract}
The hypergraph Ramsey number of two $3$-uniform hypergraphs $G$ and $H$, denoted by $R(G,H)$, is the least integer~$N$ such that every red-blue edge-coloring of the complete $3$-uniform hypergraph on $N$ vertices contains a red copy of $G$ or a blue copy of $H$. 

The Fano plane $\FF$ is the unique 3-uniform hypergraph with seven edges on seven vertices in which every pair of vertices is contained in a unique edge. There is a simple construction showing that $R(H,\FF) \ge 2(v(H)-1) + 1.$  Hypergraphs $H$ for which the equality holds are called $\FF$-good. Conlon asked to determine all $H$ that are $\FF$-good.

In this short paper we make progress on this problem and prove that the tight path of length $n$ is $\FF$-good.
\end{abstract}

\thispagestyle{empty}
\section{Introduction}
Ramsey theory is one of the most intensively studied topics in combinatorics. Given two hypergraphs $G$ and $H$, $R(G,H)$ denotes the hypergraph Ramsey number of $G$ and $H$. That is, $R(G,H)$ is the least integer such that any red-blue edge-coloring of the complete 3-uniform hypergraph on that many vertices contains a red $G$ or a blue $H$ as a subhypergraph. The existence of $R(G,H)$ is guaranteed by Ramsey's theorem \cite{Ramsey}. However, estimating or even determining these parameters is often a difficult problem. 

In this short paper we will determine exactly the Ramsey number of the tight path and the Fano plane. This problem is the first progress on a question asked at AIMS workshop on hypergraph Ramsey problems in 2015 by Conlon~\cite{conlonques}. He defined a graph $G$ to be $t$-good if $$R(G,K_t)=(t-1)(v(G)-1)+1,$$ and more generally defined a (hyper)graph $G$ to be $H$-good if 
\begin{equation}
 R(G,H)=(\chi(H)-1)(v(G)-1)+\sigma(H) \label{H-good}
\end{equation}
where $\sigma(H)$ is the size of the smallest color class in any $\chi(H)$-coloring of $H$. His intuition behind these definitions was that $H$-good graphs tend to be poor expanders. Denoting the Fano plane by $\FF$, he asked which hypergraphs are $\FF$-good.

Our result belongs to the rare family of problems where an exact bound is found. The origin of these questions are graph Ramsey problems (Bondy and Erd\H{o}s \cite{Erdos}). Recently, the Ramsey number of the cycle and the clique (Keevash, Long and Skokan \cite{Keevash}), and the Ramsey number of the clique and the hypercube (Griffiths, Morris, Fiz Pontiveros, Saxton, Skokan \cite{Pontiveros}) have been determined. For similar results we refer the interested reader to the excellent recent survey~\cite{ramseysurvey}.

Denote by $\FF$ the Fano plane, i.e.\ the unique $3$-uniform hypergraph with seven edges on seven vertices in which every pair of vertices is contained in a unique edge. In this paper all hypergraphs are $3$-uniform. Let $\Ti$ be the tight path on $n$ vertices, i.e.\ it contains distinct vertices $v_1,v_2,\ldots,v_n$ and edges $e_1,e_2,\ldots,e_{n-2}$ where $e_i=\{v_i,v_{i+1},v_{i+2}\}$. 
\begin{theorem} \label{mainthm}
 There exists $n_0\in \mathbb{N}$ such that for any $n\geq n_0$, 
  we have $R(\Ti,\FF )=2n-1$. 
\end{theorem}
The lower bound $R(\Ti,\FF )\geq 2n-1$ follows easily by the following folklore construction. Write the vertex set of the complete $3$-uniform hypergraph on $2n-2$ vertices $K_{2n-2}^{(3)}$ as the disjoint union of two sets $A$ and $B$ with $|A|=|B|=n-1$. Color all $3$-edges that are fully contained in either $A$ or in $B$ red, and all other edges blue. Observe that as $|A|,|B|<n$ there is no red $\Ti$ in this coloring. Since the chromatic number $\chi(\FF)=3$, in every copy of $\FF$ there is one edge that is fully contained in either $A$ or $B$, hence this coloring cannot contain a blue copy of $\FF$ either. This establishes that $R(\Ti,\FF)>2n-2$. The main contribution of our work is to establish the upper bound $R(\Ti,\FF )\leq 2n-1$. 

Let $\Ci$ be the tight cycle on $n$ vertices, i.e.\ it contains distinct vertices $v_1,v_2,\ldots,v_n$ and edges $e_1,e_2,\ldots,e_n$ with $e_i=\{v_i,v_{i+1},v_{i+2}\}$ where $v_{n+1}:=v_1$ and $v_{n+2}:=v_2$. 
Theorem~\ref{mainthm} also holds when one replaces $\Ti$ with $\Ci$.

\begin{theorem} \label{mainthm2}
There exists $n_0\in \mathbb{N}$ such that for any $n\geq n_0$, we have $R(\Ci,\FF )=2n-1$. 
\end{theorem}
We choose not to present the proof of Theorem~\ref{mainthm2} since the proof is almost the same as the proof of Theorem~\ref{mainthm} and differs only in some technicalities which do not give the reader more insight into the methods used.

Theorem~\ref{mainthm} claims that $\Ti$ is $\FF$-good. We might reverse the question and ask which hypergraphs $H$ are $\Ti$-nice. Here we say that $H$ is $G$-nice, if  \eqref{H-good} holds.

 The organization of the paper is the following. In Section 2 we will show that Theorem~\ref{mainthm} is sharp in the sense that the Ramsey number increases when one adds enough edges to the tight path. In Section 3 we will give some definitions and basic tools which will be needed for the proof. The proof itself will be given in Section 4. It will consist of two cases, which will be handled separately. 


\section{Sharpness example}
The following example shows that Theorem~\ref{mainthm} is best possible in the following sense: Let $P'$ be the $3$-uniform hypergraph obtained from the tight path $\Ti$ by adding three edges $\{v_2, v_3, v_6\}$,  $\{v_1, v_2, v_5\}$ and $\{v_1, v_4, v_6\}$. We claim that Ramsey number of $P'$ and $\FF$ is bigger than the Ramsey number of $\Ti$ and $\FF$.

\begin{lemma}
$R(P',\FF )> 2n$
\end{lemma} 

\begin{proof}
Assume, for simplicity, that $n$ is divisible by 3. Take three sets $A$, $B$, $C$ of size $2n/3$ each. Color by red all triples of vertices $\{x,y,z\}$ such that either $x,y\in A$ and $z\in A\cup B$, or $x,y\in B$ and $z\in B\cup C$, or $x,y\in C$ and $z\in A\cup C$. All other triples are colored blue. A short case analysis shows there is no blue Fano plane $\FF$. Let us assume there is an embedding of a red $P'$. The first three vertices of such an embedding cannot come from different sets $A,B$ and $C$. Without loss of generality, let $A$ be the set which contains at least two of them. The only way to embed a red copy of $\Ti$ is to use all vertices of $A$ and $n/3$ vertices of $B$. Since between 2 vertices from $B$ there has to be at least 2 vertices from $A$, the only way for an embedding of $\Ti$ to start is with the first 6 vertices having the following patterns: $AABAAB$, $ABAABA$, $ABAAAB$, $BAABAA$, $BAAABA$ or $BAAAAB$. However, regardless of which pattern we use, the resulting red tight path cannot be extended to a red copy of $P'$: one of $\{v_2, v_3, v_6\}$,  $\{v_1, v_2, v_5\}$ and $\{v_1, v_4, v_6\}$ would be of the form $BBA$ and therefore blue.
\end{proof}


\section{Preparations}
Our starting point in the proof of Theorem~\ref{mainthm} will be an upper bound on the off-diagonal hypergraph Ramsey numbers. We choose to use an upper bound from~\cite{ConFoxSud}, but any weaker bound would suffice.

\begin{theorem}[Conlon--Fox--Sudakov \cite{ConFoxSud}] \label{Suda}
There exists $C>0$ so that for every integer $s\geq 4$ and sufficiently large $t$,  
$$R(K_s^{(3)},K_t^{(3)})\leq 2^{Ct^{s-2}\log t}.$$
\end{theorem}
 Let the hyperedges of $\HH:=K_{2n-1}^{(3)}$ be two-colored with colors red and blue, without a blue $\FF$. In the proof we will make use of the following definitions.

 \begin{definition} \label{butterfly}
 Given two disjoint sets $A,B$ of vertices in $\HH$, we say four vertices $a_1,a_2\in A$, $b_1,b_2\in B$ form a red \emph{butterfly} if up to symmetry the two hyperedges $a_1a_2b_1$ and $a_2b_1b_2$ are red in $\HH$. 

 We denote with $|\overrightarrow{AB}|_r$ ($|\overrightarrow{AB}|_b$) the number of red (blue) hyperedges in $\HH$ of the form $ab_1b_2$ with $a\in A, b_1,b_2\in B$.
 Given three disjoint sets $A,B,C$ of vertices in $\HH$, we denote with $|ABC|_r$ ($|ABC|_b$) the number of red (blue) hyperedges of the form $abc$ with $a\in A,b\in B, c\in C$. 

For $W \subset V(\HH), v\notin \HH$, denote $G_{v,W}^{blue}$ the blue link graph of $v$ in $W$, i.e.\ the graph on $W$ with $ab$ being an edge iff $abv$ is blue in $\HH$. Analogously, $G_{v,W}^{red}$ defines the red link graph. 

For $t\in \mathbb{N}$, we define the complete directed bipartite graph $\overrightarrow{K}_{t,t}$ to be the directed graph on vertex set $A \cup B$ with $|A|=|B|=t$, $A$ and $B$ disjoint, and the arc set $\{ab | \ a\in A, b\in B \}$. 
 \end{definition}

The following theorem is a directed version of the K{\"o}v{\'a}ri--S{\'o}s--Tur{\'a}n Theorem \cite{Sos}.  
 \begin{theorem} \label{finding directed}
Let $t,m\in \mathbb{N}$. Define $D$ to be a digraph with vertex set $A \cup B$, where $A$ and $B$ are disjoint, and $|A|=|B|=m$. If the number of arcs from $A$ to $B$ is at least $C'm^{2-1/t}$ for $C'$ being a constant large enough only depending on $t$, then $D$ contains a directed $\overrightarrow{K}_{t,t}$ from $A$ to $B$.     
 \end{theorem}
 
\noindent
 The following tool consisting of the next two Lemmas will be used multiple times in the main proof. 
\begin{lemma}\label{claimbluetri2}
Let $A,B,C \subseteq V(\HH)$ such that $|A|=|B|=|C|= m$ and assume that there are at most 1000 vertex-disjoint red butterflies connecting each pair of the three sets $A,B,C$.

 Then there exists an absolute constant $t>0$ such that for $m$ big enough  $$|\overrightarrow{AB}|_r ,|\overrightarrow{BC}|_r,|\overrightarrow{CA}|_r \leq m^{3-1/t} \quad \text{ or } \quad |\overrightarrow{BA}|_r ,|\overrightarrow{CB}|_r,|\overrightarrow{AC}|_r \leq m^{3-1/t}.$$ 
\end{lemma}

\begin{proof}
 Removing at most $4000$ vertices from each set, we end up with sets $A_1 \subset A ,A_2 \subset B,A_3\subset C$ so that there are no red butterflies connecting them. Note that if for some $a\in A_i$, $b\in A_j$ the pair $ab$ is not contained in the same red butterfly, then either all hyperedges $\{abx : x\in A_i\}$ or all hyperedges $\{aby : y\in A_j\}$ are blue.

Create a digraph $\overrightarrow{G}$ with $V(\overrightarrow{G})=A_1 \cup A_2 \cup A_3$ as follows. We have $\overrightarrow{uv}\in E(\overrightarrow{G})$ if there are some $i,j$ with $i\neq j, u\in A_i, v \in A_j$ and the set of hyperedges $\{uvy : y\in A_j\}$ is all blue. Note that some edges might be oriented in both ways.

Let $t$ be the bipartite Ramsey number for $K_{4,4}$. That is, $t$ is the least integer such that every two-coloring of the edges of $K_{t,t}$ contains a monochromatic copy of $K_{4,4}$. Note that by Irving~\cite{Irving}, we have  $t\leq  48$. 

Suppose between $A_1$ and $A_2$, both the left and right density is at least $C'm^{-1/t}$ for $C'$ being a constant large enough, so by Theorem~\ref{finding directed} we can find complete directed $\overrightarrow{K}_{t,t}$ in both directions. That is, there are sets $D_1, D_2 \subset A_1$ and  $E_1, E_2 \subset A_2$ of sizes $t$ each such that the edges $\{\overrightarrow{ab}: a\in D_1, b\in E_1\}$ and $\{\overrightarrow{ba}: a\in D_2, b\in E_2\}$ are all present in $\overrightarrow{G}$. 

Pick an $x\in D_1$ and set $N_{x}:=N^+(x)\cap A_3$, where $N^+(x)$ means the out-neighborhood of $x$ in $\overrightarrow{G}$. Suppose for contradiction that $|N_{x}|\geq t$. Then in $\overrightarrow{G}$ there is a directed $K_{4,4}$ between $N_{x}$ and $E_1$; that is, we can find $S\subset N_{x}$ and $T\subset E_1$ such that $|S|,|T|\geq 4$ and (w.l.o.g) all edges of the form $\{\overrightarrow{st}: s\in S, t\in T\}$ are present in $\overrightarrow{G}$. But that is impossible, as letting $S=\{s_1,\ldots,s_4\}$, $T=\{t_1,\ldots,t_4\}$ the hyperedges $\{xs_1s_2,xt_1t_2,xt_3t_4,s_1t_1t_4,$ $s_1t_2t_3,s_2t_1t_3,s_2t_2t_4\}$ are all blue in $\HH$ and form a Fano plane.

This implies that every vertex in $D_1$ has an out-neighbourhood in $A_3$ of size at most $t$. Repeating the same argument after replacing $D_1$ by $E_2$, we get that every vertex in $E_2$ has an out-neighbourhood in $A_3$ of size at most $t$. So by removing at most $2t^2$ vertices from $A_3$ we get a set $A_3'$ of the property that all edges from $A_3'$ to $D_1$ and all edges from $A_3'$ to $E_2$ are present in $\overrightarrow{G}$. By the choice of $t$, we can find sets $W_1\subset D_1$ and $W_2\subset E_2$ of sizes at least four such that w.l.o.g. $(W_1,W_2)$ forms a directed $K_{4,4}$. Hence we have found a $4$-blowup of a transitive triangle. But this is impossible, as letting $W_1=\{v_1,v_2,v_3,v_4\}$, $W_2=\{w_1,w_2,w_3,w_4\}$, $a \in A_3'$ the hyperedges $\{av_1v_2,aw_1w_2,aw_3w_4,v_1w_1w_4,v_1w_2w_3,v_2w_1w_3$$,v_2w_2w_4  \}$ are all blue in $\HH$ and form a Fano plane. 

This proves that in $\overrightarrow{G}$ between $A_1$ and $A_2$, in one of the directions the density has to be less than $C'm^{-1/t}$. Repeating this argument for the other two pairs, we get that between any pair of sets from $A_1,A_2,A_3$, in one of the directions the density has to be less than $C'm^{-1/t}$ whereas the density in the other direction has to be at least $1-C'm^{-1/t}$. The majority orientation forms a transitive triangle or an oriented 3-cycle. Suppose now that the majority orientation forms a transitive triangle. Pick four vertices from each set at random. Then the probability that the 12 vertices do not form a $4$-blowup of the transitive triangle is at most $48C'm^{-t}$. Therefore, there exists a $4$-blowup of a transitive triangle in $\overrightarrow{G}$, giving a blue Fano plane in $\HH$. Thus, the majority orientation has to form a 3-cycle. 

W.l.o.g. let $A_1 \rightarrow A_2 \rightarrow A_3 \rightarrow A_1$ be the majority orientation in $\overrightarrow{G}$.
Then the density between $A_1$ and $A_2$ is at least $1-C'm^{-1/t}$ in $\overrightarrow{G}$, then this implies $|\overrightarrow{A_1A_2}|_r \leq C'm^{3-1/t}$ and as we only deleted at most 12000 vertices in the beginning, also  $|\overrightarrow{AB}|_r \leq 2C'm^{3-1/t}$. Repeating this argument for the other pairs gives us
$ |\overrightarrow{BC}|_r \leq 2C'm^{3-1/t}$  and  $ |\overrightarrow{CA}|_r \leq 2C'm^{3-1/t}.$ By choosing  $t$ slightly bigger we get rid of the constant $2C'$ for large enough $m$.   
\end{proof}

\noindent
Lemma~\ref{claimbluetri2} can be improved in the following way. 
\begin{lemma}\label{claimbluetri}
Let $A,B,C \subseteq V(\HH)$ such that $|A|=|B|=|C|= m$ and assume that there are at most 1000 vertex-disjoint red butterflies connecting each pair of the three sets $A,B,C$.

 Then there exists an absolute constant $t$ such that for $m$ big enough  $$|\overrightarrow{AB}|_r ,|\overrightarrow{BC}|_r,|\overrightarrow{CA}|_r, |\overrightarrow{BA}|_b ,|\overrightarrow{CB}|_b,|\overrightarrow{AC}|_b \leq m^{3-1/t}$$
 or
 $$|\overrightarrow{BA}|_r ,|\overrightarrow{CB}|_r,|\overrightarrow{AC}|_r,|\overrightarrow{AB}|_b ,|\overrightarrow{BC}|_b,|\overrightarrow{CA}|_b \leq m^{3-1/t}.$$ 
\end{lemma}
\begin{proof}
Applying Lemma~\ref{claimbluetri2}, we get a positive constant $t'$ such that w.l.o.g. 

$$ |\overrightarrow{AB}|_r ,|\overrightarrow{BC}|_r,|\overrightarrow{CA}|_r \leq m^{3-1/t'}.$$

\noindent
Let $t=3t'$. 
For the sake of contradiction, say that $|\overrightarrow{AC}|_b \geq m^{3-1/t}.$ Define  $$Z_1:= \left\{v\in A \bigg| \ e(G_{v,B,b}) \geq \frac{99}{100} \binom{m}{2}\right\} \quad \text{ and } \quad  Z_1':=\left\{v \in A \bigg| \ e(G_{v,C,b})\geq \frac{1}{2}m^{2-1/t} \right\}.$$

\noindent
Then $ |A \setminus Z_1| \leq  600 m^{1-1/t'}$, as otherwise $|\overrightarrow{AB}|_r \geq 600 m^{1-1/t'} \frac{1}{100} \binom{m}{2} > 2m^{3-1/t'}$. Also $|Z_1'|\geq 800m^{1-1/t'}$, as otherwise  $|\overrightarrow{AC}|_b < 800m^{1-1/t'} m^2 + \frac{1}{2} m^{2-1/t} m \leq m^{3-1/t}.$ Thus, one can choose a vertex $v\in Z_1 \cap Z_1'$.  Let 
$$Y_1:= \left\{ w\in B \bigg| e(G_{w,C,r})\leq 10m^{2-1/t'} \right\}.$$

\noindent
Then $|Y_1|\geq 4/5 m$ as otherwise $|\overrightarrow{BC}|_r\geq 2m^{3-1/t'}$. Because of the size of $Y_1$, $G_{v,B,b}$ has to contain an edge inside $Y_1$. Let $w_1w_2$ be such an edge. The number of 4-tuples $(a,b,c,d)$ of distinct vertices $a,b,c,d\in C$ with $ab,cd\in E(G_{v,C}^{blue})$ is at least

$$ \sum_{ab\in E(G_{v,C,b})} \left( e(G_{v,C,b})-deg(a)-deg(b) \right) \geq e(G_{v,C,b}) (e(G_{v,C,b})-2m) \geq \frac{1}{5}m^{4-2/t}. $$

The number of 4-tuples $(a,b,c,d)$ of distinct vertices $a,b,c,d\in C$ with $ad \notin E(G_{w_1,C}^{blue})$ or $bc \notin E(G_{w_1,C}^{blue})$ is at most $e(G_{w_1,C}^{red})m^2+m^2 e(G_{w_1,C}^{red}) \leq 20m^{4-1/t'}$. Similarly, the number of 4-tuple $(a,b,c,d)$ of distinct vertices $a,b,c,d\in C$ with $ac \notin E(G_{w_2,C}^{blue})$ or $bd \notin E(G_{w_2,C}^{blue})$ is at most $20m^{4-1/t'}$. Since $20m^{4-1/t'} + 20m^{4-1/t'} < \frac{1}{5}m^{4-2/t},$ there exists $a,b,c,d \in C$ such that $ab,cd\in E(G_{v,C}^{blue}); ad,bc\in E(G_{w_1,C}^{blue})$ and $ac,bd \in E(G_{w_2,C}^{blue})$. Thus, the hyperedges $vw_1w_2,vab,vcd,w_1ad,$ $w_1bc,w_2ac,w_2bd$ form a blue Fano plane; a contradiction, therefore we conclude that $|\overrightarrow{AC}|_b \leq m^{3-1/t}.$ Similarly, we get $|\overrightarrow{CB}|_b \leq m^{3-1/t}$ and $|\overrightarrow{BA}|_b \leq m^{3-1/t}.$

\end{proof}

\section{Proof of Theorem~\ref{mainthm}}
\subsection{Set up of the proof}

For the sake of contradiction, assume that there is a red-blue edge-coloring of $\HH:=K_{2n-1}^{(3)}$ without a blue $\FF$ and without a red $\Ti$. Fix such a coloring. 
\noindent
Let $\eps>0$ be a sufficiently small constant and assume that $n$ is sufficiently large. Set $$m=\left\lceil\eps \sqrt[5]{\frac{\log n}{\log \log n}}~\right\rceil.$$
Observe that $m^5\log m\leq \frac{\eps^5}{5}\log n$, hence we have by Theorem~\ref{Suda} $$R(K_m^{(3)},K_7^{(3)})\leq 2^{Cm^5 \log m}\leq n^{\eps^4}.$$

Since $\HH$ contains no blue $\FF$, it cannot contain a blue $K_7^{(3)}$ and we conclude that it contains a red $K_m^{(3)}$, call it $D_1$. Set $\HH_1:=\HH\setminus V(D_1)$ and find a red $K_m^{(3)}$, call it $D_2$, in $\HH_1$. Repeating this process, setting $\HH_{i+1}:=\HH_i\setminus V(D_i)$, we can find a red copy of $K_m^{(3)}$ in $\HH_{i+1}$, calling it $D_{i+1}$, as long as $|V(\HH_i)|\geq n^{\eps^4}$. At the end of this process we end up with a collection of vertex-disjoint red $K_m^{(3)}$-s $D_1,D_2,\ldots,D_d$, and a set $J$ of remaining vertices with $|J|\leq n^{\eps^4}$.

Create a graph $G_1$ with $V(G_1)=\{D_1,\ldots,D_d\}$, by connecting $D_i,D_j$ if in $\HH$ there are at least $1000$ vertex-disjoint red butterflies between them. The vertices of $G_1$ will be called blobs. The next two lemmas give information on the structure of $G_1$.

\begin{lemma} \label{K_4 lemma}
The complement of $G_1$ contains no $K_4$. 
\end{lemma}
\begin{proof}
For the sake of contradiction, assume that there are 4 blobs $A_1,A_2,A_3$ and $A_4$ which form a $K_4$ in the complement of $G_1$. 
Define a directed graph $D$ with vertex set $V(D)=\{A_1,A_2,A_3,A_4\}$ and an edge from blob $A_i$ to $A_j$ ($i \neq j$) iff $|\overrightarrow{A_iA_j}|_r  \leq m^{3-1/t}$ with $t$ from Lemma~\ref{claimbluetri}.
Applying Lemma~\ref{claimbluetri} on all subsets of size 3 of the 4 blobs gives that every edge in $D$ is oriented in exactly one direction. This means that $D$ is a tournament. However, a tournament on 4 vertices contains a transitive triangle and Lemma~\ref{claimbluetri} says this cannot happen. 
\end{proof}

\begin{lemma} \label{paths}
$G_1$ has one of the following forms:

\begin{itemize}
    \item [(i)] $V(G_1)=\{A_1,\ldots,A_a,B_1,\ldots,B_b\}$ such that $A_1,\ldots,A_a$ and $B_1,\ldots,B_b$ form vertex-disjoint paths or
    \item [(ii)]$V(G_1)=\{A_1,\ldots,A_a,B_1,\ldots,B_b,C_1,\ldots,C_c\}$ such that $A_1,\ldots,A_a$; $B_1,\ldots,B_b$ and $C_1,\ldots,C_c$ form vertex-disjoint paths.
\end{itemize}
\end{lemma}
\begin{proof}
Let $A_1,\ldots,A_a$ be a longest path in $G_1$. If $V(G_1) = \{ A_1,\ldots,A_a \}$ then one can find a red $\Ti$ in $\HH$ just by jumping from red blob to red blob along the path using the red butterflies. Let $B_1,\ldots,B_b$ be a longest path in $G_1$ on the vertices $V(G_1) \setminus \{A_1,\ldots,A_a\}$. If $V(G_1) = \{ A_1,\ldots,A_a,B_1,\ldots,B_b \}$ we are in case (i). Otherwise we can take the longest path $C_1,\ldots,C_c$ in $V(G_1) \setminus \{A_1,\ldots,A_a, B_1,\ldots,B_b\}$.  In this case $V(G_1)= \{ A_1,\ldots,A_a,B_1,\ldots,B_b,C_1,\ldots,C_c\}$ as otherwise any blob in $V(G_1) \setminus \{ A_1,\ldots,A_a,B_1,\ldots,$ $B_b,C_1,\ldots,C_c\}$ would form a $K_4$ in the complement of $G_1$ together with $A_1,B_1$ and $C_1$. This is not possible by Lemma~\ref{K_4 lemma}. 

\end{proof}
\noindent
In the next two Subsections the two cases from Lemma~\ref{paths} will be handled separately. The strategy is to build a long red tight path using these two or three blocks. When a long path say starting in $A_1$ and ending in $A_a$ is found, it is clear that some of the vertices inside the block can be used to close the cycle.  

\subsection{The two paths case}
\label{2paths}

In this case $G_1$ is the vertex-disjoint union of two paths, i.e.\ in $\HH$ we have vertex-disjoint red $K_m$-s $\{A_1,\ldots,A_a,B_1,\ldots,B_b\}$ and a set $J$ of junk vertices with $|J|\leq n^{\eps^4}$. For every $i,j$ 
there are at least $1000$ red butterflies between $A_i$ and $A_{i+1}$ and also between $B_j$ and $B_{j+1}$. Slightly abusing notation, let $P_1=\cup_i A_i$ and $P_2=\cup_j B_j$. Note that if $|P_i|\geq n$ for some $i\in\{1,2\}$ then we can embed the tight path $P_n^t$ into $P_i$ just by walking through each blob and jumping from blob to blob by using the hyperedges from the red butterflies. We know that $|P_1|+|P_2|+n^{\eps^4}\geq |V(\HH)|=2n-1$. So $n-n^{\eps^4}\leq|P_i|\leq n-1$ for $i=1,2$.

\begin{definition}\label{triple triangle}
A red triple triangle between $A_i$ and $B_j$ is a set of vertices $w,x,y,z\in A_i$ and $v\in B_j$ (or $w,x,y,z\in B_j$ and $v\in A_i$) so that $wxv,xyv,yzv$ is red in $\HH$. 
\end{definition}
\noindent
Observe that when we have a red triple triangle between $A_i$ and $B_j$ we can find a red tight path of length $m+1$ by swallowing one additional vertex from $B_j$ using the red triple triangle. If there is no red triple triangle between two blobs then there also have to be few red hyperedges between the blobs. 
\begin{lemma} \label{triple}
If there is no red triple triangle between $A_i$ and $B_j$, then  $|A_iB_j|_r+|B_jA_i|_r\leq 20m^2$. 
\end{lemma}
\begin{proof}
Pick any vertex $v\in B_j$ and consider its red link graph in $A_i$. If $v$ is not in a red triple triangle, then the red link graph does not contain a path of length $3$, hence the number of edges in this link graph is at most $10m$. So the number of red hyperedges between $B_j$ and $A_i$, assuming that there are no red triple triangles, is at most $20m^2$. 
\end{proof}

\begin{lemma} \label{decomp}
$V(\HH)$ can be decomposed as $V(\HH)=A \cup B \cup J'$ with $|A|,|B|\geq n-n^{\eps^3}, |J'|\leq n^{\eps^3}$ such that there are at most $500n^3 /m$ blue hyperedges inside $A$ and respectively in $B$.
\end{lemma}
\begin{proof}
Consider the bipartite graph $G_2$, with vertex sets $\{A_1,\ldots,A_a\}$ and $\{B_1,\ldots,B_b\}$. Connect $A_iB_j$ by an edge iff between $A_i$ and $B_j$ there is a red triple triangle in $\HH$. Let $M$ be a largest matching in $G_2$. Then we can embed into $\HH$ a tight red path of length $|P_i|+|M|/2$ for some $i\in\{1,2\}$, because at least half of the triple triangles represented by edges from the matching have to go in the same direction. 
In particular since $|P_i|\geq n-n^{\eps^4}$ for $i=1,2$ this implies $|M|\leq 2n^{\eps^4}$. Put all blobs covered by $M$ into $J$ and get a new rubbish set $J'$. We will have $|J'|\leq |J|+2|M|m\leq n^{\eps^3}$ vertices. \\
The subgraph of $G_2$ on the blobs which have not been removed spans an independent set. Let $A$ be the set of vertices in $P_1$ which have not been removed and let $B$ be the vertices in $P_2$ which have not been removed. The following argument shows that for three different blobs $A_1',A_2',A_3'$ from $\{A_1,\ldots,A_a\}$ which have not been removed, $|A_1'A_2'A_3'|_b \leq 400m^2$. For contradiction, assume there are more than that many blue hyperedges. Take a blob $B_i$ which has not been removed from $\{B_1,\ldots,B_b\}$. By Lemma~\ref{triple} $$|\overrightarrow{A_1'B_i}|_r, |\overrightarrow{A_2'B_i}|_r,|\overrightarrow{A_3'B_i}|_r,|\overrightarrow{B_iA_1'}|_r,|\overrightarrow{B_iA_2'}|_r,|\overrightarrow{B_iA_3'}|_r \leq 20m^2.$$ Picking at random one vertex each of $A_1',A_2'$ and $A_3'$, and 4 vertices from $B_i$, these vertices do not form a blue Fano plane with probability at most $1-400m^{-1}+ 6 \cdot 50m^{-1}$, thus, there has to exist a blue Fano plane. We conclude $|A_1'A_2'A_3'|_b \leq 400m^2$. Therefore there are at most $400m^2(n/m)^3 + m^3(n/m)^2 \leq 500n^3 /m$ blue hyperedges inside $A$. Similarly, this holds for $B$.  
\end{proof}

 \begin{definition}  \label{special}
Call a vertex $v\in B$ \textit{special} if $e(G_{v,A}^{red})\geq \frac{1}{5} \eps^5 \binom{n}{2}$. Similarly, call a vertex $v\in A$ \textit{special} if  $e(G_{v,B}^{red})\geq \frac{1}{5} \eps^5 \binom{n}{2}$.
 \end{definition}

\begin{lemma} \label{embedd}
Let $V(\HH)=A \cup B \cup J'$ be the decomposition from Lemma~\ref{decomp}. If there are at least $nm^{-1/20}$ \textit{special} vertices in $A$ or $B$, then one can find a red $\Ti$ in $\HH$.
\end{lemma}
\begin{proof}
Suppose there are at least $nm^{-1/20}$ \textit{special} vertices in w.l.o.g. $B$. We now show that we can absorb enough of these \textit{special} vertices from $B$ to find a tight red path of length $n$.  
Let $a,b,c,d \in A,v\in B$. A tuple $(c,d)$ is called \textit{reachable} from $(a,b)$ if both $abc$ and $bcd$ are red. Further, a tuple $(c,d)$ is called \textit{reachable} from $(a,b)$ via $v$ if all $abv,bvc,vcd$ are red.   
A tuple $(a,b)$ is called open if there exists at most $n^2m^{-1/4}$ tuples $(c,d)$ with $c,d\in A$ such that
$(c,d)$ is not \textit{reachable} from $(a,b)$.  Define $O$ to be the set of all open tuples. As there are at most $500n^3 /m$ blue hyperedges inside $A$, $|O|\geq |B|(|B|-1) - n^2m^{-1/4}$. 
Call a tuple $(a,b)$  \textit{good} for $v\in B$ if there exists at least $ \eps^{100} n^2$ tuples $(c,d), c,d\in A$ such that $(c,d)$ is \textit{reachable} from $(a,b)$ via $v$. Denote $Good(v)$ the set of all tuples being \textit{good} for $v$. 
For $v\in B$ \textit{special},  $|Good(v)|\geq \eps^{100} n^2$, because otherwise the number of $P_4$'s in $G_{v,B,b}$ would be at most $2\eps^{100} n^4$. However, since $e(G_{v,B,b})\geq \frac{1}{5} \eps^5 \binom{n}{2}$, the number of $P_4$'s in $G_{v,B}^{blue}$ is more than $2\eps^{100} n^4$.   

We will now walk along the red hyperedges step by step adding in each step 5 vertices to the tight path. Let $v_1,v_2,\ldots$ be the \textit{special} vertices in $B$. Let $(a_1,b_1) \in Good(v_1)$ and $(c_1,d_1)$ be an open tuple such that $(c_1,d_1)$ is \textit{reachable} from $(a_1,b_1)$ via $v_1$. We begin the walk with $a_1,b_1,v_1,c_1,d_1$. Now take a look at step $i$. Assume we already have defined $a_1,b_1,v_1,c_1,d_1,a_2,\ldots,a_{i-1},b_{i-1},v_{i-1},c_{i-1},d_{i-1}$ with $(c_{i-1},d_{i-1})$ being open. Pick $(a_i,b_i) \in Good(v_i)$ such that  $(a_i,b_i)$ is \textit{reachable} from $(c_{i-1},d_{i-1})$. For $i\leq nm^{-1/20}$, this is possible, because

$$ |Good(v_i)| -n^2m^{-1/4} - 5ni \geq \frac{\eps^{100}}{2} n^2.$$

Now pick $(c_{i},d_{i}) \in O$ such that $(c_{i},d_{i})$ is \textit{reachable} from $(a_{i},b_{i})$ via $v_1$. For $i\leq nm^{-1}$, this is possible, because

$$ \eps^{100} n^2 - n^2m^{-1/4} - 5ni \geq \frac{\eps^{100}}{2} n^2.$$

Now enlarge the path with $a_ib_iv_ic_id_i$. After $i\leq nm^{-1/20}$ steps we end up with a path of length $5nm^{-1/20}$ such that the last two vertices form an open tuple. Now, we just keep picking open tuples and walk from an open tuple to an open tuple. We can keep doing this until at most $nm^{-1/10}$ vertices are not used inside $A$. This means we found a tight red path of length at least

$$n-n^{\eps^4}- nm^{-1/10} + nm^{-1/20}\geq n.$$
\end{proof}

\begin{lemma} \label{decomp 2}
Let $V(\HH)=A \cup B \cup J'$ be the decomposition from Lemma~\ref{decomp}. If there are at most $nm^{-1/20}$ \textit{special} vertices in $A$ and $B$, then the vertex set $V(\HH)$ can be decomposed into $V(\HH)=A' \cup B' \cup J''$ with $|A'|,|B'|\geq n-\eps n, |J''|\leq \eps n$ such that $\HH[A']$ and $\HH[B']$ are entirely red, for all $v\in A'$ $e(G_{v,B'}^{red})\leq \frac{1}{5} \eps^5n^2$ and for all $v\in B'$ $e(G_{v,A'}^{red})\leq \frac{1}{5} \eps^5n^2.$

\end{lemma}
\begin{proof}
We can remove all \textit{special} vertices from $A$ and $B$ and add them to the junk set $J'$. So we obtain $A',B',J''$  so that for each $v\in A'$ the red link graph in $B'$ has at most $\frac{1}{5} \eps^5n^2$ edges, for each $w\in B'$ the red link graph of $w$ in $A'$ has at most $\frac{1}{5} \eps^5n^2$ edges and $|J''|\leq |J'|+2nm^{-1/20} \leq \eps n.$ 

\noindent
Suppose $abc$ is a blue hyperedge in $A'$. Let $G_{a,B'}^{blue},G_{b,B'}^{blue},G_{c,B'}^{blue}$ be the blue link graphs in $B'$. By the previous observation $e(G_{a,B'}^{blue} \cap G_{b,B'}^{blue} \cap G_{c,B'}^{blue}) \geq \frac{9}{10} \binom{n}{2}$ and thus $G_{a,B'}^{blue} \cap G_{b,B'}^{blue} \cap G_{c,B'}^{blue}$ contains a $K_4$. The four vertices from the $K_4$ together with $a,b,c$ contain a blue copy of the Fano plane in $\HH$. Hence $\HH[A']$ is entirely red. The same holds for $\HH[B']$. 
\end{proof}

\begin{lemma}\label{adding junk} 
In the setting of Lemma~\ref{decomp 2}, we can decompose $J''=J_1 \cup J_2$ such that for all $v\in J_1$ $e(G_{v,A'}^{blue}) \leq\eps n^2$  and all $v\in J_2$ $e(G_{v,B'}^{blue}) \leq\eps n^2$. 
\end{lemma}
\begin{proof}
Let $V(\HH)=A' \cup B' \cup J''$ be the decomposition from Lemma~\ref{decomp 2}. We actually will prove that if a rubbish vertex $v \in J''$ has $e(G_{v,A'}^{blue})\geq \eps n^2$, then it cannot have a blue hyperedge into $B'$. Indeed suppose there are $a,b\in B'$ with $abv$ blue. Then both $a$ and $b$ are part of at least $\binom{|A'|}{2}-\frac{1}{5} \eps^5n^2$ blue hyperedges with the other two vertices being in $A'$. Therefore there are at least $\binom{
|A'|}{2}-\frac{2}{5}\eps^5n^2$ pairs $(c,d)$ with $c,d\in A'$ and $cda,cdb$ blue, call this set of edges $S$. Now suppose $v$ leads at least $\eps n^2$ blue hyperedges into $A'$, i.e.\ $e(G_{v,A'}^{blue})\geq \eps n^2$ and hence $G_{v,A'}^{blue}$ contains a path $P$ of length $\eps^2 n$. The restriction of $S$ onto the vertex set of $P=\{p_1,p_2,\ldots\}$ contains at least $(1-\eps)\binom{|P|}{2}$ edges and hence contains four vertices $p_i,p_{i+1},p_j,p_{j+1}$ with $p_ip_j,p_ip_{j+1},p_{i+1}p_j,p_{i+1}p_{j+1}$ in $S$. Then $ap_ip_{j+1},avb,ap_jp_{i+1},p_{j+1}bp_{i+1},p_{j+1}vp_j, p_ivp_{i+1},p_ip_jb$ form a blue Fano plane. Hence we can split up the junk set $J''=J_1 \cup J_2$ such that for all $v\in J_1$ $e(G_{v,A'}^{blue}) \leq\eps n^2$  and all $v\in J_2$ $e(G_{v,B'}^{blue}) \leq\eps n^2$.
\end{proof}

\begin{lemma}\label{adding junk2} 
In the setting of Lemma~\ref{decomp 2}, we can find a red $\Ti$ in $\HH$. 
\end{lemma}
\begin{proof}
Let $J''=J_1 \cup J_2$ be the decomposition of the junk vertices from Lemma~\ref{adding junk}. Set $A^*=A' \cup J_1$ and $B^*=B' \cup J_2$. Then either $|A^*|\geq n$ or $|B^*|\geq n$. W.l.o.g. $|A^*|\geq n$. Now one can find a red tight path of length $n$ inside $A^*$. Let $v_1,v_2,\ldots$ be the vertices from $J_1$. Call a tuple $(a,b),a,b,\in A'$ \textit{good*} for $v\in J_1$ if the red link graph of $v$ in $A'$ contains at least $\frac{9}{10}n^2$ tuples $(c,d), c,d\in A'$ such that $abv,bvc,vcd$ are red in $\HH$. Since the blue link graph of $v$ contains at most $\eps n^2$ edges, for each $v\in J_1$ the number of \textit{good*} tuples is at least $\frac{9}{10}n^2$. Now start with an arbitrary \textit{good*} tuple $(a_1,b_1)$ for $v_1$. The start of the walk is $a_1,b_1,v_1$. Now assume we already have chosen $a_1,b_1,v_1,\ldots,a_{i-1},b_{i-1},v_{i-1}$ such that $(a_{i-1},b_{i-1})$ is \textit{good*} for $v_{i-1}$. Take a \textit{good*} tuple $(a_i,b_i)$ for $v_i$ of unused vertices such that $b_{i-1}v_{i-1}a_i, v_{i-1}a_ib_i$ are red. This is possible for all $i\leq \eps n$, because

$$  \frac{9}{10}n^2 - \frac{1}{10}n^2 - 3in >0. $$

Enlarge the path by $a_ib_iv_i$. After all vertices from $J_1$ are used, just walk through $A'$ until all vertices in $A'$ are used. This is possible, because all hyperedges inside $A'$ are red. Thus, we find a red tight path of length $|A^*|\geq n$. 
\end{proof}

\subsection{The three paths case}

Let $A,B,C$ with $A=\{A_1,A_2,\ldots, A_a\}, B=\{B_1,\ldots,B_b\}$ and $C=\{C_1,\ldots,C_c\}$ be the three paths in $G_1$. There cannot be an edge between blobs of different paths. Otherwise one can split up each blob in two blobs of equal size (if $m$ is odd one vertex ends up in $J$) in such a way that when one constructs the graph of all new blobs, where two blobs are connected with each other when they have at least 100 disjoint red butterflies between them, and then builds the longest paths that one can handle this situation in the same way as in case $p=2$. Hence instead of three paths we shall have two paths and we handled this already in Subsection \ref{2paths}.

Using Lemma~\ref{claimbluetri} it follows that w.l.o.g.  $|\overrightarrow{A_iB_j}|_r,|\overrightarrow{B_jC_k}|,|\overrightarrow{C_kA_i}|_r \leq m^{3-1/t}$ for all $i,j,k$. Let $P_1=\cup A_i , P_2= \cup B_i$ and $P_3= \cup C_i$. Clearly, $|P_1|,|P_2|,|P_3| <n$ as otherwise one could find a red tight path of length $n$ just by going through a blob and then jumping to the next by using a red butterfly and so on. 

\begin{definition}\label{gra}For $X,Y \subset V(\HH)$ and $0<x<1$, denote $G(X,Y,x)$ the graph with vertex set $X$, and $ab$ is an edge iff the number of red hyperedges $abc$ with $c$ from $Y$ is at least $x|Y|$.
\end{definition}

\begin{lemma}\label{equal size}
There exists a constant $C'$ such that 
$$\frac{2}{3}n -C'nm^{-1/t} \leq |P_i|\leq \frac{2}{3}n +C'nm^{-1/t}$$ 
\noindent
for $i=1,2,3$. 
\end{lemma}
\begin{proof}
Let $A_i\in A, B_i \in B$. We now will show that one can find a red tight path of length at least $3m/2-4000m^{1-1/t}$ just using $A_i$ and $B_i$ starting with two vertices from $A_i$, ending with two vertices from $A_{i}$, not using two vertices being part of a butterfly to $A_{i-1}$ and not using two other vertices being part of a butterfly to $A_{i+1}$. Consider the graph $G(A_i,B_i,0.99)$. The number of vertices in this graph with degree at most $0.9m$ is at most $2000m^{1-1/t}$ as otherwise $|\overrightarrow{B_iA_i}|_r > 2000m^{1-1/t} \cdot 0.1m \cdot 0.01m \cdot 0.5 =m^{3-1/t}$.
Let $A'\subseteq A_i$ be the set of all vertices of degree at least $0.9m$ not containing the two vertices being part of a butterfly to $A_{i-1}$ and not containing the two vertices being part of a butterfly to $A_{i+1}$. Then $|A'|\geq m- 2000m^{1-1/t}-4$, $G(A',B_i,0.99)$ has minimum degree at least $0.8m$ and thus there exists a Hamiltonian path $v_1,v_2,\ldots,v_{A'}$ in this graph. After every second vertex in this path we will now add a vertex from $B_i$ to find a long red tight path in $\HH$. Assume we already found the tight red path $v_1,v_2,w_1,v_3,v_4,w_2,\ldots,v_{2i-1},v_{2i}$. Then we can pick a vertex $w_i\in B_i$  which has not been used yet and such that $v_{2i-1}v_{2i}w_i,v_{2i}w_i,v_{2i+1},w_i,v_{2i+1}v_{2i+1}$ are red for $i< |A'|/2$. This is possible because $m-0.01m-0.01m-0.01m-i>0$. Thus, we can find a red tight path of length at least $3/2 (m-2001m^{1-1/t})\geq 3m/2-4000m^{1-1/t}$. If $|P_1|\leq |P_2|$, then we can find a tight red path of length at least $3/2 |P_1|-5000nm^{-1/t}$ by the following argument. Jump from blob to blob in $A$ using the vertices from the butterflies and always absorbing the vertices from the index corresponding blob in $B$.
When we are done with all blobs in $A$ we stop and have found a red tight path of length at least 

$$ \left(\frac{3}{2}m - 4000m^{1-1/t} \right) \left\lfloor \frac{|P_1|}{m} \right\rfloor \geq \left(\frac{3}{2}m - 4000m^{1-1/t} \right) \left( \frac{|P_1|}{m} -1\right) \geq \frac{3}{2}|P_1| - 5000nm^{-1/t}.  $$

If $|P_1|\geq |P_2|$, then we can find a tight red path of length at least $ (|P_1|+ 1/2|P_2|)-5000nm^{-1/t}$ by the following argument. Jump from blob to blob in $A$ using the vertices from the butterflies and always absorbing the vertices from the index-corresponding blob in $B$.
When we are done with all blobs in $B$ we go back to $A$ and walk through the remaining blobs from $A$ using the butterflies. So we can find a tight red path of length at least

$$ \left(\frac{3}{2}m - 4000m^{1-1/t} \right)   \left\lfloor\frac{|P_2|}{m} \right\rfloor + |P_1|-|P_2| \geq |P_1|+\frac{|P_2|}{2}-5000nm^{-1/t}. $$

We will now show that the sizes of the blocks $P_1,P_2,P_3$ is at most $2/3n+C'nm^{-1/t}$ and at least $2/3n-C'nm^{-1/t}$ for an absolute constant $C'$. W.l.o.g. let $P_1$ be a biggest block. If $|P_1|\leq 2/3n+30000nm^{-1/t}$, then also $|P_2|,|P_3| \leq 2/3n+30000nm^{-1/t}$, but since $|P_1|+|P_2|+|P_3|= 2n-1-|J|,$ we also get $|P_1|,|P_2|,|P_3| \geq 2/3n-60001nm^{-1/t}$. Assume $|P_1|\geq 2/3n+30000nm^{-1/t}$. If $|P_2|\geq 2(n-|P_1|)+10000nm^{-1/t}$, then we find a red tight path of length at least  $$(|P_1|+ \frac{1}{2}|P_2|)-5000nm^{-1/t}\geq n+5000nm^{-1/t}.$$ Otherwise,  $|P_2|\leq 2(n-|P_1|)+10000nm^{-1/t}$. Then, $|P_1|+|P_2| \leq 2n-|P_1|+10000nm^{-1/t}$ and thus $|P_3|\geq |P_1|-10001nm^{-1/t}\geq 2/3n +19999nm^{-1/t}.$ 
But now we can find a red tight path of length at least $\frac{3}{2}|P_1| - 5000nm^{-1/t}>n$. This shows that there exists a constant $C'$ such that 
$$\frac{2}{3}n -C'nm^{-1/t} \leq |P_i|\leq \frac{2}{3}n +C'nm^{-1/t}$$
\noindent
for $i=1,2,3$. 
\end{proof}

\begin{lemma}\label{entirely red}
$V(\HH)$ can be decomposed as $V(\HH)= P_1'' \cup P_2''\cup P_3''\cup J$ such that $|P_1''|=|P_2''|=|P_3''|$; $|J|\leq C''nm^{-1/t}$ for some constant $C''>0$ and each of $\HH[P_1''],\HH[P_2''],\HH[P_3'']$ is entirely red. 
\end{lemma}
\begin{proof}
The number of vertices $v \in A_1$ with $e(G_{v,B_1}^{blue})\leq 29/30 \binom{m}{2}$ is at most $30m^{1-1/t}$ as otherwise $|\overrightarrow{A_1B_1}|_r\geq m^{3-1/t}$. This means one can move at most $30m^{1-1/t} (n/m) \leq 30nm^{-1/t}$ vertices $v$ from $P_1$ to $J$ (and obtain $P_1' \subseteq P_1$) such that all vertices in $P_1'$ satisfy $e(G_{v,B_1}^{blue})\geq 29/30 \binom{m}{2}$. Now assume there is a blue hyperedge $v_1,v_2,v_3$ inside $P_1'$. Since $e(G_{v_1,B_1}^{blue} \cap G_{v_2,B_1}^{blue} \cap G_{v_3,B_1}^{blue}) \geq 27/30 \binom{m}{2}$, $G_{v_1,B_1}^{blue} \cap G_{v_2,B_1}^{blue} \cap G_{v_3,B_1}^{blue}$ contains a $K_4$. These 4 vertices together with $v_1,v_2,v_3$ form a blue Fano plane. Thus $P_1'$ is entirely red. Repeating this cleaning procedure for $P_2$ and $P_3$ one ends up with entire red blocks $P_1',P_2',P_3'$ and a rubbish set $J'$ of size at most $100nm^{-1/t}$. Considering that the blocks had roughly equal size, we can remove a few more vertices from the blocks and end up with entirely red blocks $P_1'',P_2'',P_3''$ of equal size and a rubbish set $J''$ of size at most $C''nm^{-1/t}$ vertices with $C''$ being an absolute constant.     
\end{proof}

\begin{lemma}\label{structure}
Let $V(\HH)= P_1'' \cup P_2''\cup P_3''\cup J$ be the decomposition from Lemma~\ref{entirely red}. Then  $|P_1''P_2''P_3''|_r \leq 7n^{3-1/t}$.
\end{lemma}
\begin{proof}
 Applying Lemma~\ref{claimbluetri} gives w.l.o.g. that    $$|\overrightarrow{P_1''P_2''}|_r ,|\overrightarrow{P_2''P_3''}|_r,|\overrightarrow{P_3''P_1''}|_r, |\overrightarrow{P_2''P_1''}|_b ,|\overrightarrow{P_3''P_2''}|_b,|\overrightarrow{P_1''P_3''}|_b \leq n^{3-1/t}.$$

 Assume $|P_1''P_2''P_3''|_r \geq 7n^{3-1/t}$. Pick $v_1,w_1\in P_1'',v_2,w_2,x_2\in P_2'', v_3 \in P_3''$ uniformly at random. The hyperedge $v_1w_1v_2$ is blue or not a proper hyperedge with probability at most $2n^{-1/t}$. Similarly, $v_2w_2v_3$ and $v_3w_2x_2$ is blue or not an hyperedge each with probability at most $2n^{-1/t}$. The hyperedge $v_1v_2v_3$ is blue with probability at most $1-7n^{3-1/t}$. Thus, the probability that one of the hyperedges $v_1w_1v_2,v_1v_2v_3,v_2v_3w_2,v_3w_2x_2$ is blue is at most $1-7n^{3-1/t}+6n^{3-1/t}<1$. Thus, there exists $v_1,v_2,v_3,w_1,w_2,w_3$ such that all the hyperedges $v_1w_1v_2,v_1v_2v_3,v_2v_3w_2,v_3w_2x_2$  are red. Now one can find a red tight path of length at least $|P_1''|+|P_2''|+1\geq n$ by first going through all vertices in $P_1''$ besides $v_1$ and $w_1$, then going along $w_1v_1v_2v_3w_1w_2$ and then through all vertices in $P_2''$. Recall that all hyperedges inside $P_1'', P_2''$ or $P_3''$ are red.  
\end{proof}

\begin{lemma}\label{structure2}
There exists a decomposition of the vertices of $\HH$ into $V(\HH)= P_1^\dagger \cup P_2^\dagger \cup P_3^\dagger$ with $ 0.66 \leq |P_i^\dagger| \leq 0.67$ for $i=1,2,3$ such that all graphs $G(P_1^\dagger,P_2^\dagger,0.98)$, $G(P_2^\dagger,P_3^\dagger,0.98)$ and $G(P_3^\dagger,P_1^\dagger,0.98)$ have minimum degree at least $0.39n$. 
\end{lemma}
\begin{proof}
Let $V(\HH)= P_1'' \cup P_2''\cup P_3''\cup J$ be the decomposition from Lemma~\ref{entirely red}. The number of vertices in $G(P_1'',P_2'',0.99)$ with degree less than $0.4n$ is at most $1500n^{1-1/t}$.  Removing at most $1500n^{1-1/t}$ vertices from each $P_1'',P_2'',P_3''$ leaves us with sets $P_1^\ast,P_2^\ast,P_3^\ast$ and a junk set $J^\ast$ such that every vertex in the graphs $G(P_1^\ast,P_2^\ast,0.98)$, $G(P_2^\ast,P_3^\ast,0.98)$ and $G(P_3^\ast,P_1^\ast,0.98)$ has minimum degree at least $0.39n$. 

We now check that every vertex $v$ in $J^\ast$ has degree at least $0.39n$ in one of the graphs $G(P_1^\ast \cup \{v\},P_2^\ast,0.98)$, $G(P_2^\ast \cup \{v\},P_3^\ast,0.98)$ and $G(P_3^\ast \cup \{v\},P_1^\ast,0.98).$ Assume this is not the case, then there exists $v\in J^\ast, X_1\subseteq P_1^\ast, X_2\subseteq P_2^\ast$ and $X_3 \subseteq P_3^\ast$ with $|X_1|=|X_2|=|X_3| \geq 0.65n-0.39n=0.26n$ such that for each $x_1 \in X_1$ there are at least $0.02 \cdot 0.65n\geq 0.01n$ many vertices $ y_2 \in P_2^\ast$ such that $vx_1y_2$ is blue. Similarly, for each $x_2 \in X_2$ there are at least $0.01n$ many vertices $ y_3 \in P_3^\ast$ such that $vx_2y_3$ is blue and for each $x_3 \in X_3$ there are at least $0.01n$ many vertices $ y_1 \in P_1^\ast$ such that $vx_3y_1$ is blue. Now pick $x_1\in X_1, x_2 \in X_2, x_3 \in X_3$ independently uniformly at random. There exist random sets $Y_1 \subset P_1^\ast,Y_2\subset P_2^\ast,Y_3\subset P_3^\ast$ with $|Y_1|=|Y_2|=|Y_3|\geq 0.01n$ such that $vx_1y_2,vx_2y_3,vx_3y_1$ for all $y_1\in Y_1, y_2\in Y_2 , y_3 \in Y_3.$ Now pick $y_1\in Y_1, y_2 \in Y_2, y_3 \in Y_3$ independently uniformly at random. The hyperedges $vx_1y_2,vx_2y_3,vx_3y_1$ are blue. As $|X_1X_2X_3|_r \leq |P_1''P_2''P_3''|_r \leq n^{3-1/t}$, $x_1x_2x_3$ is red with probability at most $4^3n^{-1/t}$. Since 
$$|Y_1X_2Y_2|_r\leq |\overrightarrow{P_1''P_2''}|_r \leq n^{3-1/t}, \quad |Y_2X_3Y_3|_r\leq |\overrightarrow{P_2''P_3''}|_r \leq n^{3-1/t}$$
 and  $|Y_3X_1Y_1|_r\leq |\overrightarrow{P_3''P_1''}|_r \leq n^{3-1/t},$
the probability that each of the hyperedges $y_1x_2y_2,$ $y_2x_3y_3,y_3x_1y_1$ is red is at most $C^\ast n^{-1/t}$ for an absolute constant $C^\ast$. Thus, with positive probability  $vx_1y_2,vx_2y_3,$ $vx_3y_1,x_1x_2x_3,y_1x_2y_2,y_2x_3y_3,y_3x_1y_1$ form a blue Fano plane. We therefore can assume that every vertex $v \in J^\ast$ has degree at least $0.39n$ in one of the graphs $G(P_1^\ast \cup \{v\},P_2^\ast,0.98)$, $G(P_2^\ast \cup \{v\},P_3^\ast,0.98)$ and $G(P_3^\ast \cup \{v\},P_1^\ast,0.98).$ Thus every vertex from $J$ can be added to $P_1^\ast$ or $P_2^\ast$ or $P_3^\ast$ such that one obtains blocks $P_1^\dagger,P_2^\dagger,P_3^\dagger$ ($ 0.66 \leq |P_i^\dagger| \leq 0.67$ for $i=1,2,3$) with $P_1^\dagger \cup P_2^\dagger \cup P_3^\dagger=[2n-1]$ in such a way that afterwards all graphs $G(P_1^\dagger,P_2^\dagger,0.98)$, $G(P_2^\dagger,P_3^\dagger,0.98)$ and $G(P_3^\dagger,P_1^\dagger,0.98)$ have minimum degree at least $0.39n$. 
\end{proof}

\begin{lemma}\label{structure3}
In the setting of Lemma~\ref{structure2} we can find a red $\Ti$. 
\end{lemma}
\begin{proof}
Since $P_1^\dagger \cup P_2^\dagger \cup P_3^\dagger =[2n-1]$ one of the blocks has size at least $2n/3$. W.l.o.g.\ $|P_1^\dagger|\geq 2n/3$. The minimum degree of $G(P_1^\dagger,P_2^\dagger,0.98)$ assures that this graph contains a Hamiltonian path. Label such a path $a_1,a_2,a_3,\ldots,a_{|P_1^\dagger|}$. In order to find a tight path of length $n$, we will add after every second vertex in this path a vertex from $P_2^\dagger$. Assume we already have found $a_1,a_2,b_1,a_3,a_4,b_2,\ldots,a_{2i-1}a_{2i}$ then we can choose $b_i$ from $P_2^\dagger$ which has not been used so far such that $a_{2i-1}a_{2i}b_i$, $a_{2i}b_ia_{2i+1}$ and $b_ia_{2i+1},a_{2i+2}$ is red, because $|P_2^\dagger|-0.02|P_2^\dagger|-0.02|P_2^\dagger|-0.02|P_2^\dagger|- i >0$ for $i<0.94|P_2^\dagger|$ and thus especially for $i<0.5n$. Hence, we can embed a red tight path of length at least $2n/3 + n/3=n$.   
\end{proof}

\end{document}